\documentclass{amsart}
\usepackage[margin=3.5 cm]{geometry}
\usepackage{latexsym}
\usepackage{amsfonts}
\usepackage{amssymb,amscd}
\usepackage{amsmath}
\usepackage{graphics, color}
\usepackage{hyperref}

\def\r{\mathcal R}
\def\R{\mathbb R}
\def\s{\mathbb S}
\def\C{\mathbb C} 

\def\F{{\mathrm F}}
\def\L{\mathrm L}

\def\z{{\bf z}}
\def\w{{\bf w}}

\def\V{\mathrm V}

\def\T{\mathbb T}

\def\X{\mathbb X}
\def\A{\mathbb A}

\def \ch{{\bf H}_{\C}}

\def\R{\mathbb R}
\def\V{\mathrm V}

\def\p{{\bf p}}

\newcommand {\tr }{{\rm tr}}
\newcommand{\SL}{\mathrm{SL}}
\newcommand{\SU}{{\mathrm{SU}}}

\newcommand{\U}{\mathrm{U}}
\newcommand{\PU}{{\mathrm {PU}}}
\def \a{\mathcal A}

\def\P{\mathbb P}

\def\M{\mathbb M}

\newcommand{\fX}{\mathfrak{X}}

\newcommand{\hm}{{\mathrm{Hom}}}

\newcommand{\fR}{\mathfrak{R}}

\newcommand{\fD}{\mathfrak{X_L}}

\def \a {{\bf a}}
\def \r  {{\bf r}}
\def \x  {{\bf x}}
\def \y  {{\bf y}}
\def \q {{\bf q}}

\newtheorem{theorem}{Theorem}[section]
\newtheorem{lemma}[theorem]{Lemma}
\newtheorem{prop}[theorem]{Proposition}
\theoremstyle{definition}
\newtheorem{definition}[theorem]{Definition}

\theoremstyle{remark}
\newtheorem{remark}[theorem]{Remark}
\numberwithin{equation}{section}
\theoremstyle{plain}
\newtheorem{acknowledgement}{Acknowledgement}

\newtheorem{corollary}[theorem]{Corollary}

\numberwithin{equation}{section}

\newcommand{\secref}[1]{Section~\ref{#1}}
\newcommand{\thmref}[1]{Theorem~\ref{#1}}
\newcommand{\lemref}[1]{Lemma~\ref{#1}}

\newcommand{\propref}[1]{Proposition~\ref{#1}}
\newcommand{\corref}[1]{Corollary~\ref{#1}}
\newcommand{\eqnref}[1]{~{\textrm(\ref{#1})}}
\newcommand{\defref}[1]{Definition~(\ref{#1})}

\begin{document}
\title[Conjugation Orbits of Loxodromic pairs in $\SU(n,1)$]{Conjugation Orbits of Loxodromic Pairs in $\SU(n,1)$}
\author[Krishnendu Gongopadhyay  \and Shiv Parsad]{Krishnendu Gongopadhyay \and
 Shiv Parsad}
\address{Indian Institute of Science Education and Research (IISER) Mohali,
 Knowledge City,  Sector 81, S.A.S. Nagar 140306, Punjab, India}
\email{krishnendug@gmail.com, krishnendu@iisermohali.ac.in}
\address{Indian Institute of Science Education and Research (IISER) Bhopal, 
Bhopal Bypass Road, Bhauri,
Bhopal 462 066,
Madhya Pradesh, India} 
\email{parsad.shiv@gmail.com}
\date{\today}
\begin{abstract}

Let $\ch^n$ be the $n$-dimensional complex hyperbolic space and $\SU(n,1)$ be the (holomorphic) isometry group. An element $g$ in $\SU(n,1)$ is called \emph{loxodromic} or \emph{hyperbolic} if it has exactly two fixed points on the boundary $\partial \ch^n$.  We classify $\SU(n,1)$ conjugation orbits of pairs of loxodromic  elements in $\SU(n,1)$. 
\end{abstract}
\maketitle
\section{Introduction}

Let $\ch^n$ be the $n$-dimensional complex hyperbolic space. The group $\SU(n,1)$ acts by the holomorphic isometries on $\ch^n$.   An element of $\SU(n,1)$ is called \emph{hyperbolic} or \emph{loxodromic} if it has exactly two fixed points on the boundary $\partial \ch^n$ of the complex hyperbolic space.  

Let $\F_2=\langle x, y \rangle$ be a two-generator free group.  Let  $\fX(\F_2, \SU(n,1))$ denote the orbit space $\hm(\F_2, {\rm SU}(n,1))/{\rm SU}(n,1)$ of the conjugation action of $\SU(n,1)$ on the space $\hm(\F_2, \SU(n,1))$ of faithful representations of $\F_2$ into $\SU(n,1)$.  Let $\fD(\F_2, \SU(n,1))$ denote the subset of $\fX(\F_2, \SU(n,1))$ consisting of representations $\rho$ such that both $\rho(x)$ and $\rho(y)$ are loxodromic elements in $\SU(n,1)$ having no common fixed point. A problem of geometric interest is to 
 parametrize this subset $\fD(\F_2, \SU(n,1))$.  The motivation for doing this is the construction of Fenchel-Nielsen coordinates in the classical Teichm\"uller space that is built upon a parametrization of discrete, faithful, and totally loxodromic representations in $\fD(\F_2, {\rm \SL}(2, \R))$.  This is rooted back to the classical works of Fricke \cite{f} and Vogt \cite{v} from whom it follows that a non-elementary two-generator free subgroup of ${\rm SL}(2, \R)$ is determined  up to conjugation by the traces of the generators and their product, see Goldman \cite{gold2} for an exposition.

 The space $\fD(\F_2, \SU(n,1))$ contains the discrete, faithful, and  totally loxodromic or type-preserving  representations. These are curious families of representations and has not been well-understood even in the case $n=2$. We refer to the surveys \cite{ppf}, \cite{sc}, \cite{will3} and the references therein for an up to date account of the investigations in this direction. 

For notational convenience, an element in $\fD(\F_2, \SU(n,1))$ will be called a `loxodromic generated representation', or simply, a `loxodromic representation'.  Most of the existing works to understand $\fD(\F_2, \SU(n,1))$ is centered around the case $n=2$, though it would  be interesting to generalize some of above mentioned works for $n>2$. A starting point for this could be the classification of pairs of elements in ${\rm SU}(n, 1)$ up to conjugacy. In other words, the problem would be to determine a representation in $\fD(\F_2, \SU(n,1))$. 

 To do this, following classical invariant theory, one approach is to obtain this classification using trace invariants like the coefficients of the characteristic polynomials and their compositions.  In low dimensions, this approach gives some understanding of the loxodromic pairs. Will \cite{will, will2} classified the loxodromic pairs in $\SU(2,1)$. Will's classification is built upon the work of Lawton \cite{law}, also see  \cite{wen},  who obtained trace parameters for elements in $\fX(\F_2, \SL(3, \C))$. It follows from these works that an irreducible  loxodromic representation in $\fX(\F_2, \SU(2,1))$  is determined by the traces of the generators and the traces of three more compositions of the generators. In an attempt to generalize this work,  Gongopadhyay and Lawton \cite{gl} have classified polystable pairs (that is, pairs whose conjugation orbit is closed in the character variety) in $\SU(3,1)$ using 39 real parameters. At the same time, it has been shown that the real dimension of the \hbox{smallest}  possible system of such real parameters to determine any polystable pair is 30. Explicit set of trace coordinates for $\fX(\F_2, \SU(n,1))$ is not available in the literature for higher values of $n$. A set of such coordinates may be obtained for arbitrary $n$ using the fact that $\SU(n,1)$ is a real form of $\SL(n+1, \C)$.  The trace coordinates in $\fX(\F_2, \SL(n+1,\C))$ may be obtained from the work of Lawton \cite{law2}, however,  minimal set of  trace parameters along with relations between them, is still unknown except for a few lower values of $n$,  eg. \cite{ds, do}.

Using the geometry of the boundary points, there is another approach to classify pairs of loxodromic elements in low dimensions. This was used by Wolpert \cite{sw} to parametrize surface group representations into $\SL(2, \R)$. In the complex hyperbolic setting, Parker and Platis \cite{pp} obtained a classification of the loxodromic pairs in $\SU(2,1)$ and used that to parametrize  loxodromic representations into $\SU(2,1)$. Independently, Falbel \cite{falbel}, also see \cite{fap}, has taken a viewpoint using configuration of four points on the boundary and classified the loxodromic pairs in $\SU(2,1)$ up to conjugacy. Both Parker and Platis, and Falbel have associated a point on an algebraic variety, that along with the traces of the elements  classified the loxodromic pairs. 
Cunha and Gusevskii \cite{cugu1} associated traces of the elements and $\SU(2,1)$-orbits of ordered tuples of fixed points to achieve another classification of  loxodromic representations.  Cunha and Gusevskii's work also gave an explicit embedding of the space $\fD(\F_2, \SU(2,1))$ into an affine space. 

We have asked this question for $\SU(3,1)$ and obtained partial results in this direction in \cite{gp2}, where we  have  classified generic pairs of loxodromics in $\SU(3,1)$ called `non-singular'.  We introduced new parameters analogous to the Kor{\'a}nyi-Reimann cross ratios, but involving fixed points and polar points, to achieve this classification in \cite{gp2}.   In view of this, the main result in \cite{gp2} provides a smaller number of $15$ real dimensional coordinates that is enough to determine generic pairs of loxodromic elements in ${\rm SU}(3,1)$, and using that a Fenchel-Nielsen type coordinate system was given on the `non-singular' components of the character variety. 

 In this paper, we classify $\SU(n,1)$-conjugation orbits of pairs of loxodromic elements in $\SU(n,1)$ using a  geometric approach.  The key intuitive idea to do this is to view a pair of loxodromics as a pair of `moving orthonormal frames' and then attach a tuple of boundary points to it that corresponds to the `moving chains' defining the pair. The collection of such tuples of boundary points form a topological space that comes from the $\SU(n,1)$ configuration space $\M(n, m)$ of ordered $m$-tuples of points on $\partial \ch^n$. The space $\M(n,m)$ has been described by Cunha and Gusevskii in \cite{cugu1}.

 Given a loxodromic element $A$, we choose a normalized eigenbasis that corresponds to an orthonormal frame of $\C^{n,1}$. Now to a specified point $p$ on the complex hyperbolic line joining the fixed points, we choose a polar eigenvector $x$. This gives a chain spanned by the $p$ and $x$, and we mark it with a point. 
These marked points, along with the fixed points, determine the eigenframe of a loxodromic element that we started with. Given a pair, we do this for both elements in the pair. This choice is not canonical. However, given the algebraic multiplicity of a loxodromic pair, the orbit of such points under the group action induced by the change of eigenframes is canonical, and we denote the space of such orbits by $\mathcal{OL}_n$.  The point on this space that corresponds to a pair, is called the \emph{canonical orbit} of the pair. This space has a topological structure that comes from the space $\M(n,m)$.  The advantage of associating canonical boundary points to a loxodromic pair is that it enables us to obtain a parametric description of arbitrary elements in $\fD(\F_2, \SU(n,1))$. 
\begin{theorem}\label{mthm}
Let $\rho$ be an element in $\fD(\F_2, \SU(n,1))$, $\F_2=\langle x, y \rangle$. Then $\rho$ is determined uniquely by $\tr(\rho(x)^i)$, $\tr(\rho(y)^i)$, $1\leq i \leq \lfloor (n+1)/2 \rfloor$, and the  canonical orbit of $(\rho(x), \rho(y))$. 
\end{theorem}
In other words,  what we shall prove is the following. 
\begin{theorem}\label{thm2}
Let $(A, B)$ be a loxodromic pair in $\SU(n,1)$. Then $(A, B)$ is determined up to conjugation in $\SU(n,1)$,  by the the traces ${ \tr}({A^i})$, ${\tr}( B^i)$,  $1\leq i \leq \lfloor (n+1)/2 \rfloor$, and the  canonical orbit of $(A, B)$ on $\mathcal{OL}_n$. 
\end{theorem}
Further if we choose one representative from each orbit, it associates numerical conjugacy invariants. Given $[(A, B)] \in \fD(\F_2, \SU(n,1))$, we may choose a  representative $\mathfrak p$ of $(A, B)$ from the corresponding class in $\mathcal{OL}_n$. Such a choice associates a point $[\mathfrak p]$ on the space $\M(n, 2n+2)$ to the class $[(A, B)]$. This class $[\mathfrak p]$ is called a \emph{reference orbit} of $(A, B)$. After such a choice, it is now possible to associate conjugacy invariants like angular invariants and  cross ratios to the pair $(A, B)$. Such a  chosen set of cross ratios are termed as `reference cross ratios'. With this notion, we have the following. 
\begin{theorem}\label{mt2}
 Let $(A,B)$ be a loxodromic pair in $\SU(n,1)$. Then $(A,B)$ is determined up to conjugation by $\SU(n,1)$,  by the traces $ {\tr}(A^i),~{\tr}( B^i)$,  $1\leq i \leq \lfloor (n+1)/2 \rfloor$, the angular invariant and the reference  cross ratios. 
\end{theorem}
We emphasize here that the choice of these numerical invariants depends on the choice of the representative of the orbit class. We do not know how to make it canonical.  Using this viewpoint, one can project the loxodromic elements as a tuple of the boundary of the complex hyperbolic space. We hope this viewpoint will be useful in the understanding of the loxodromic representations.  Further we discuss special classes of pairs for whom the choices of the conjugacy invariants are canonical. The generic pairs of $\SU(3,1)$ classified in \cite{gp2} were called `non-singular'. Here we generalize that notion to $\SU(n,1)$. We show that for the non-singular pairs in $\SU(n,1)$ the canonical orbit projects to a point, and hence one can canonically associate numerical invariants to classify them.

After discussing some preliminary notions in \secref{prel}, we review loxodromic elements of $\SU(n,1)$  in \secref{lox}. The detailed notions involved in the above theorems are discussed in the following sections. We associate tuples of boundary points to a loxodromic element in \secref{canp}.  Specifying such an association, we prove \thmref{mt2} in \secref{tol}. By the work of Cunha and Gusevskii \cite{cugu1}, this also gives us conjugacy invariants to be associated to such loxodromic pairs. In \secref{ntol}, we make this association well-defined by associating the whole orbit of points  to a pair and prove \thmref{thm2}.  In \secref{good}, we construct examples of two classes of loxodromic pairs, including the non-singular ones,  for which the associated boundary tuples define a single orbit and association of the conjugacy invariants is canonical. 
 
\section{Preliminaries}\label{prel} 
\subsection{Complex Hyperbolic Space} 
Let $\V=\C^{n,1}$ be the complex vector space $\C^{n+1}$ equipped with the Hermitian form of signature $(n,1)$ given by $$\langle\z,\w\rangle=\w^{\ast}H\z=z_{1}\bar w_{n+1}+z_2\bar w_2+\cdots+z_n\bar w_n
+z_{n+1}\bar w_1,$$
where $\ast$ denotes conjugate transpose. The matrix of the Hermitian form is given by 
\begin{center}
$H=\left[ \begin{array}{ccc}
            0 & 0 & 1\\
           0 & I_n & 0 \\
 1 & 0 & 0 \\
          \end{array}\right]$
\end{center}
If $H^{\prime}$ is any other $(n+1)\times (n+1)$ Hermitian matrix with signature $(n,1)$, then there is a $(n+1)\times(n+1)$ matrix $C$ so that $C^{\ast}H^{\prime}C=H$.  

\medskip We consider the following subspaces of $\C^{n,1}:$
$$\V_{-}=\{\z\in\C^{n,1}:\langle\z,\z \rangle<0\}, ~ \V_+=\{\z\in\C^{n,1}:\langle\z,\z \rangle>0\},~\V_{0}=\{\z\in\C^{n,1}:\langle\z,\z \rangle=0\}.$$
Let $\P:\C^{n,1}-\{0\}\to\C \P^n$ be the canonical projection onto complex projective space. Then complex hyperbolic space $\ch^n$ is defined to
be $\P (\V_{-})$. The ideal boundary $\partial\ch^n$ is $\P( \V_{0})$. The canonical projection of a vector $\z\in \V_{-}$ is given by 
\hbox{$\P(\z)=(z_1/z_{n+1},\ldots,z_n/z_{n+1})$}. Therefore we can write $\ch^n=\P(\V_{-})$ as
$$\ch^n=\{(w_1,\ldots,w_n)\in\C^n \ : \ 2\Re(w_1)+|w_2|^2+\cdots+|w_n|^2<0\}.$$  
This gives the Siegel domain model of $\ch^n$. 
There are two distinguished points in $\V_{0}$ which we denote by  $\bf{o}$ and $\bf\infty$ given by 
$$\bf{o}=\left[\begin{array}{c}
               0\\\vdots\\0\\1\\
              \end{array}\right],
 ~~\infty=\left[\begin{array}{c}
               1\\0\\\vdots\\0\\
              \end{array}\right].$$\\
Then we can write $\partial\ch^n=\P(\V_{0})$ as
$$\partial\ch^n-\infty=\{(z_1,\ldots,z_n)\in\C^n:2\Re(z_1)+|z_2|^2+\cdots+|z_n|^2=0\}.$$
 Conversely,  given a point $z$ of $\ch^n=\P(\V_{-})\subset\C \P^n$,  we may lift $z=(z_1,\ldots,z_n)$ to a point $\z$ in $\V_{-}$, called the \emph{standard lift} of $z$, 
 by writing in non-homogenous coordinates as
 $$\z=\left[\begin{array}{c}
                z_1\\\vdots\\z_n\\1\\
               \end{array}\right].$$
\subsection{Isometries}   
 Let ${\rm U}(n,1)$ be the group of matrices which preserve the Hermitian form $\langle .,.\rangle$. Each such matrix $A$ satisfies the relation $A^{-1}=
H^{-1}A^{\ast}H$,  where $A^{\ast}$ is the conjugate transpose of $A$. The holomorphic isometry group of  $\ch^n$ is the projective unitary group 
$${\rm PSU}(n,1)={\rm SU }(n,1)/\{ I,\omega I,\ldots,\omega^n I\},$$ $\omega=\cos(2\pi/(n+1))+i\sin(2\pi/(n+1))$.  It is often more convenient to lift to the $(n+1)$-fold cover ${\rm SU }(n,1)$ to look at the action of the isometries. 

\subsection{Cartan's angular invariant}\label{cai}
Let $z_1,~z_2,~z_3$ be three distinct points of $\partial\ch^n$ with lifts $\z_1,~\z_2$ and $\z_3$ respectively. Cartan's angular invariant is defined as 
follows:
$$\A(z_1,~z_2,~z_3)=arg(-\langle \z_1,\z_2 \rangle \langle \z_2,\z_3 \rangle \langle \z_3,\z_1 \rangle).$$
The angular invariant is invariant under ${\rm SU }(n,1)$ and independent of the chosen lifts. Angular  invariant determines any triples of distinct points on $\partial\ch^n$ up to ${\rm SU }(n,1)$ equivalence. Also the following holds. For proofs see  \cite{gold}.
\begin{prop}\label{cra}
Let $z_1$, $z_2$, $z_3$ be three distinct points of $\partial \ch^n$ and let $\A=\A(z_1, z_2, z_3)$ be their angular invariant. Then $\A \in [-\frac{\pi}{2}, \frac{\pi}{2}]$.  $\A=\pm \frac{\pi}{2}$ if and only if $z_1$, $z_2$, $z_3$ lie on the same chain. $\A=0$ if and only if $z_1$, $z_2$, $z_3$ lie on a totally real totally geodesic subspace. 
\end{prop}
\subsection{The Kor{\'a}nyi-Reimann cross ratio}\label{cr}
Given a quadruple of distinct points $(z_1, z_2, z_3, z_4)$ on $\partial \ch^n$, their Kor{\'a}nyi-Reimann cross ratio $\X(z_1, z_2, z_3, z_4)$ is defined by
$$\X(z_1, z_2, z_3, z_4)=\frac{\langle {\bf z}_3, {\bf z}_1 \rangle \langle {\bf z}_4, \bf z_2 \rangle} { \langle {\bf z}_4, {\bf z}_1\rangle \langle   {\bf z}_3, {\bf z}_2 \rangle},$$
where, for $i=1,2,3,4$,  ${\bf z}_i$ is the lift of $z_i$. It can be seen easily that $\X$ is invariant under ${\rm SU}(n,1)$ action and independent of the chosen lifts of $z_i$'s. For more details on cross ratios see~\cite{gold}. 

Let $p = (z_1,\ldots,z_m)$ be an ordered $m$-tuple of distinct points of $\partial\ch^n$. Following Cunha and Guseveskii \cite{cugu1}, we associate to $p$ the following
numerical invariants:
$$\A(p)=\A(z_1,z_2,z_3), ~  ~\X_{2j}(p)=\X(z_1,~z_2,~z_3,~z_j)$$ 
$$\X_{3j}(p)=\X(z_1,~z_3,~z_2,~z_j) \text{ and }\X_{kj}(p)=\X(z_1,~z_k,~z_2,~z_j),$$
where $m\geq 4, 4 \leq j \leq m, 4 \leq k \leq m-1, k < j$.
\noindent Using the theory of Gram matrices, Guseveskii~\cite{cugu1} have obtained the following result.
\begin{theorem}\label{cugu}{\rm(Cunha and Gusevskii)}
Let $p = (z_1,\ldots,z_m),p' = (z_1',\ldots,z_m')$ be two ordered $m$-tuples of distinct points of $\partial\ch^n$. Suppose that for $m\geq 4, 4 \leq j \leq m, 4 \leq k \leq m-1, k < j$,
$\A(p)=\A(p'),\X_{2j}(p)=\X_{2j}(p'),\X_{3j}(p)=\X_{3j}(p'),\X_{kj}(p)=\X_{kj}(p')$. Then there exists $A\in {\rm SU}(n,1)$ such that $A(z_i)=z_i'$, $i=1, \ldots, m$. 
\end{theorem}
Cunha and Gusevskii further constructed moduli space $\M(n, m)$ of $\PU(n,1)$-configuration of ordered $m$-tuples of boundary points as a subspace of $\R^{m(m-3)+1}$. When $t<m$, we shall view $\M(n, t)$ as a subspace of $\R^{m(m-3)+1}$ embedded by the canonical inclusion map of the boundary points:
$$(p_1, \ldots, p_t) \mapsto (p_1, \ldots, p_t, 0, \ldots, 0).$$
\section{Loxodromic elements in $\SU(n,1)$}\label{lox}
The following facts are standard.
\begin{lemma}\label{lile2}
 Let $A=\begin{bmatrix} Ae_1 & \ldots & Ae_{n+1} \end{bmatrix} \in {\rm SU }(n,1)$.  For $2 \leq i\leq n$, the vector $Ae_i$ is uniquely determined by the vectors $Ae_1,\ldots,{\hat Ae_i},\ldots, Ae_{n+1}$. It is the vector orthogonal to the subspace spanned by $Ae_1,\ldots,{\hat Ae_i},\ldots, Ae_{n+1}$. 
\end{lemma}
\begin{corollary}\label{licor1}
Let $A=\begin{bmatrix} Ae_1 & \ldots & Ae_{n+1}\end{bmatrix} ,~B=\begin{bmatrix}Be_1 & \ldots & Be_{n+1}\end{bmatrix} \in {\rm SU }(n,1)$ and $C\in {\rm SU }(n,1)$ be such that $CAe_j=Be_j$ for $j\neq i, 1\leq j\leq n+1$, then $CAe_i=Be_i.$
\end{corollary}

\subsection{Loxodromics in $\SU(n,1)$}\label{sec:lox}

Let $A\in {\rm SU }(n,1)$ be  loxodromic.  Then $A$ has eigenvalues
 of the form $re^{i\theta},e^{i\phi_1},\ldots,~e^{i\phi_{n-1}},re^{i\theta}$, where $\theta, \phi_i \in (-\pi, \pi]$ for $1\leq i\leq n-1$, satisfying
 $2\theta+\phi_1+\dots+\phi_{n-1} \equiv 0 \pmod{2 \pi}$.  Let $a_A\in\partial\ch^n$ be the attractive fixed point of $A$, then any lift ${\bf a}_A$ of $a_A$ to $\V_0$ is an eigenvector of $A$ and corresponds to the 
eigenvalue  $re^{i\theta}$. Similarly,  if $r_A\in\partial\ch^n$ is the repelling fixed point of $A$, then any lift $\bf{r}_A$ of $r_A$ to $\V_0$ is an
eigenvector of $A$ with eigenvalue $r^{-1}e^{i\theta}$. For $r>1$, $\theta, \phi_i \in (-\pi, \pi]$ for $1\leq i\leq n-1$, define $E_A(r,\theta,\phi_1,\ldots,\phi_{n-1})$ as\\
\begin{equation}\label{li1}
 E_A(r,\theta,\phi_1,\ldots,\phi_{n-1})=\left[\begin{array}{ccccc}
                         re^{i\theta} & ~ &~ & ~& ~\\
                         ~ & e^{i\phi_1} & ~ & ~ & ~\\
                         ~ & ~ & \ddots & ~ & ~ \\ 
                         ~ & ~ &  ~ & e^{i\phi_{n-1}} & ~\\
                         ~ & ~ & ~ & ~ & r^{-1}e^{i\theta}\\
                        \end{array}\right].\\
\end{equation}
For $1\leq i\leq n-1$, let $\textbf{x}_{i,A}$ be an eigenvector corresponding to the eigenvalue $e^{i\phi_i}$ scaled so that 
$\langle \textbf{x}_{i,A},\textbf{x}_{i,A} \rangle=1$. Let $C_A=\left[\begin{array}{ccccc}
 \textbf{a}_A & \textbf{x}_{1,A} & \ldots & \textbf{x}_{n-1,A} & \textbf{r}_A\\
  \end{array}\right]$ be the $(n+1)\times (n+1)$ matrix, where
the lifts are chosen so that the eigenvectors form an orthonormal set, i.e.
$$\langle \a_A,\r_A \rangle=\langle \x_{i,A},\x_{i,A}\rangle =1, ~\langle \x_{i, A}, \x_{j, A}\rangle=0, ~i \neq j.$$
Then $C_A\in {\rm SU }(n,1)$ and $A=C_{A}E_{A}(r,\theta,\phi_1,\ldots, \phi_{n-1})C_{A}^{-1}$, where
$E_{A}(r,\theta,\phi_1,\ldots, \phi_{n-1})$ is given by \eqnref{li1}.
\begin{lemma} \label{lile1}
 Let $A \in \SU(n,1)$. Then $A$ has characteristic polynomial 
$$\chi_A(X)=\sum_{i=0}^{n+1}(-1)^{i}s_ix^{n+1-i},$$ where $s_0=s_{n+1}=1$ and $s_i={\bar s_{n+1-i}}$.
\end{lemma}

\begin{proof}
 Let $\Lambda=\{\lambda_1,\ldots, \lambda_{n+1} \}$ be set of eigenvalues of $A$. Using the fact that $\Lambda$ is invariant under inversion in 
 unit circle and $\text{det}(A)=1$, we have the result. 
\end{proof}

\begin{prop}\label{lipr1} 
 Two loxodromic elements in ${\rm SU}(n,1)$ are conjugate if and only if they have the same eigenvalues.
\end{prop}

\begin{corollary}\label{licor3}
 Let $A$ and $A'$ are two loxodromic elements in ${\rm {\rm SU }}(n,1)$ such that ${\text tr}(A^k)={\text tr}(A'^k)$ for $1\leq k\leq \lfloor (n+1)/2 \rfloor$ and $a_{A}=a_{A'},~
r_{A}=r_{A'}$ and $x_{i,A}=x_{i,A'}$, where $i$ ranges over $n-2$ numbers in $\{1, \ldots, n-1\}$. Then $A=A'$.  
\end{corollary}

Since every co-efficient of $\chi_A(x)$ can be expressed as a polynomial in $\tr(A)$, $\tr(A^2), \ldots$,   $\tr(A^{n-1})$, an immediate consequence of Lemma \ref{lile1} and \propref{lipr1} is the following. 
\begin{corollary}\label{cor:conjugacy}
 Two loxodromic elements $A$ and $A^\prime$ in ${\rm SU }(n,1)$ are conjugate if and only if ${\text tr}(A^k)={\text tr}(A')^k$ for $1\leq k\leq \lfloor (n+1)/2 \rfloor$.\end{corollary}

\begin{remark}
 In \lemref{lile1}, $s_{(n+1)/2}$ is real, for $n$ odd. In this case, we need $(n-1)/2$ complex parameters $s_1,\ldots,s_{(n-1)/2}$ and
  one real parameter $s_{(n+1)/2}$. When $n$ is even, we need $n/2$ complex parameters $s_1,\ldots, s_{n/2}$. In each case, we need $n$ 
  real parameters to describe a loxodromic element up to conjugacy.
\end{remark}

Note that a loxodromic element is \emph{regular} if $\chi_A(x)$ has mutually distinct roots. In other words, a loxodromic element $A$ is regular if and only if it has exactly $n+1$ fixed points on $\C \P^n$.  The following result is a part of \cite[Theorem 3.1]{gpp}. 
\begin{lemma}  \label{do}
Let $R(\tau)$ denote the resultant of the characteristic polynomial $\chi_A(x)$ and its first derivative $\chi_A'(x)$. Then $A$ is regular loxodromic in $\SU(n,1)$ if and only if $R(\tau)<0$.  
\end{lemma}
In the above,  the variable $\tau$ is given by the traces of elements: 
$$\tau=(\tr(A), \tr^2(A), \ldots, \tr^{\lfloor (n+1)/2 \rfloor}(A)).$$
As a corollary, we have:
\begin{corollary}\label{cl1}
A regular loxodromic element of $\SU(n,1)$  is completely determined by its fixed-point set $\mathfrak p(A)$ on $\C \P^n$ and its image in  the domain of traces:
$$\mathcal T=R(\tr)^{-1} (-\infty, 0)=\{\tau \in \C^{\lfloor (n+1)/2 \rfloor} \ | \ R(\tau)<0\}.$$
\end{corollary} 
In particular, this provides a well-defined correspondence of the set ${\mathcal  R}_{lox}$ of regular loxodromics in $\SU(n,1)$  onto $\partial \ch^n \times \partial \ch^n \times \underbrace{\P(\V_+) \times\cdots\times\P(\V_{+})}_{(n-1)\text{ times}}\times \mathcal T$:
$$A \mapsto (\mathfrak p(A), \tau(A)).$$
 This correspondence will be used later. It has been shown in \cite{gpp} that this correspondence is actually a smooth embedding  when $n=3$. 

\section{Eigenpoints of a loxodromic Element}\label{canp}
\subsection{Eigenpoints of a loxodromic}
Let $A$ be a loxodromic element in $\SU(n,1)$. Let $\a_A$, $\r_A$, $\x_{1, A}, \ldots, \x_{n-1, A}$ be a set of eigenvectors of $A$ chosen so that for $1 \leq i \leq n-1$,
\begin{equation} \label{n1'}
\langle \a_A,\r_A \rangle=\langle \x_{i,A},\x_{i,A}\rangle =1, ~\langle \x_{i, A}, \x_{j, A}\rangle=0, ~i \neq j.
\end{equation}
Such a choice of eigenvectors will be called a set of \emph{orthonormal eigenframe} of $A$. 

\medskip We choose a orthonormal frame of the form \eqnref{n1'} of a loxodromic element $A$. Define a set of $n+1$ boundary points associated to $A$ as follows:
\begin{equation}\label{b1}
\p_{1,A}=\a_A,~ \p_{2, A}=\r_A,~\p_{i,A}=(\a_A-\r_A)/\sqrt{2}+\x_{i-2,A}, ~3\leq i\leq n+1.
\end{equation}
We call the point $p_A=(p_{1, A}, \ldots, p_{n+1, A})$  \emph{an eigenpoint} of $A$. 

\medskip Essentially, we choose a point $p$ from the complex hyperbolic line $\langle v, v \rangle=-1$ in the projection of $\C^{1,1}$ spanned by $\{ \a_A, \r_A\}$, and then, the eigenpoints are chosen from the 1-chain that is spanned by $\p$ and $\x_{i, A}$.  So, the association of eigenpoints to $A$ depends on the choice of the projective image of an eigenframe. 

\subsubsection{Eigenspace decomposition of a loxodromic element}\label{espa}  Suppose $A$ is a loxodromic element in $\SU(n,1)$. Suppose $A$ has eigenvalues $re^{i \theta}$, $r^{-1} e^{i \theta}$, $r>1$, and $e^{ i \theta_1}, \ldots, e^{i \theta_k}$, with multiplicities $m_1, \ldots, m_k$ respectively. Then $\C^{n,1}$ has the following orthogonal decomposition  into eigenspaces:
$$\C^{n,1}=\L_A \oplus \V_{\theta_1} \oplus \ldots \oplus \V_{\theta_k},$$
here $\L_A$ is the $(1,1)$ subspace of $\C^{n,1}$ spanned by $\a_A$, $\r_A$. In the projective space, this means $A$ fixes disjoint copies of $\C\P^{m_i-1}$, for $i=1, \ldots, k$.  Thus, an orthonormal eigenframe of $A$ is determined up to  an action of $\U(m_i)$ on each $\V_{\theta_i}$.  We call $(m_1, \ldots, m_k)$ the \emph{multiplicity} of $A$. It is clear that equality of multiplicities is a necessary condition for two loxodromics to be conjugate. 

Note that the change of eigenframes amounts to a transformation of the following form that maps one frame to the other without changing the loxodromic $A$:
$$M=\begin{bmatrix} \lambda  & 0 & & \ldots&  \\ 0 & U_1 & 0 & 0 \ldots & 0\\& & 
\ddots & & \\0 & 0 & 0 &  U_k & 0 \\ 0 & 0 & 0 & 0 & \bar \lambda^{-1} \end{bmatrix},$$
where $U_i \in \U(m_i)$. 
This action is equivalent to the action of the centralizer $Z(A)$ on the eigenframes. Equivalently, this amount to conjugation of $A$ by an element of $Z(A)$.  This associates a unique $Z(A)$-orbit of eigenpoints to a  loxodromic element.

 Moreover we have the following.

\subsubsection{Congruent eigenpoints determine equivalence of eigenframes }
\begin{lemma}\label{lem:normalization}
 Let $A,~A^{'}$ be loxodromic elements in $ {\rm SU }(n,1)$ with  chosen eigenframes. Let $(p_{1, A}, \ldots, p_{n+1,A})$ and $(p_{1, A'}, \ldots, p_{n+1, A'})$ be eigenpoints of $A$ and $A'$ respectively.   Suppose that there exists $C\in {\rm SU}(n,1)$ such that $C(p_{i,A})=p_{i,A'}$,  $1\leq i \leq n$. Then $C(x_{j-2,A})=x_{j-2,A'}$ for 
  $3\leq j \leq n+1$.
\end{lemma}
\begin{proof}
Let $C(\p_{i,A})=\alpha_{i}\p_{i,A'}$ for $1\leq i\leq n$. Observe that $\langle \p_{1,A},\p_{i,A} \rangle=-1/\sqrt{2}=\langle \p_{1,A'},\p_{i,A'} \rangle$
   and $\langle \p_{2,A},\p_{i,A} \rangle=1/\sqrt{2}=\langle \p_{2,A'},\p_{i,A'} \rangle \text{ for }3\leq i\leq n$. Since $C\in \SU(n,1)$ preserve the form $\langle .,. \rangle$, we
    have $$ (-1/\sqrt{2})\alpha_1{\bar \alpha_i}=-1/\sqrt{2} \text{ and }(1/{\sqrt{2}})\alpha_{2}{\bar \alpha_i}=1/{\sqrt{2}} \text{ for } 3\leq i \leq n.$$
This implies $ \alpha_{i}={ \bar \alpha_1}^{-1}={\bar \alpha_{2}}^{-1}$ for $ 3\leq i \leq n$.  Using $ \langle C(\p_{1, A}), C(\p_{2, A})\rangle =1$ implies, $\alpha_1=\bar \alpha_2^{-1}$. Hence we must have $|\alpha_1|=1$, i.e. $\alpha_1=\bar{\alpha_1}^{-1}$. Hence 
$$C((\a_A-\r_A)/\sqrt{2}+\x_{i-2,A})={\bar \alpha_{1}}^{-1}((\a_{A'}-\r_{A'})/\sqrt{2}+\x_{i-2,A'}),$$
yields $C(x_{i-2,A})=x_{i-2,A'}$ for 
  $3\leq i \leq n$. By \corref{licor1} this implies $C(x_{n-1, A})=x_{n-1, A'}$. 
  \end{proof}

\section{Loxodromic Pairs and Reference Eigenframes}\label{tol}Throughout this paper,  given a loxodromic pair $(A, B)$ in $\SU(n,1)$, it will always be assumed that $A$ and $B$ have disjoint fixed point sets. 

\medskip Given $(A, B)$, fix a pair of associated orthonormal frames $\mathcal B=(\mathcal B_A, \mathcal B_B)$ so that
 $$\langle \a_A,\r_A \rangle=\langle \a_B,\r_B\rangle=\langle \a_A,\a_B \rangle=\langle \x_{i,A},\x_{i,A}\rangle =\langle \x_{i,B},\x_{i,B}\rangle= 1.$$
Such a normalized pair of eigenframes will be called an \emph{eigenframe} of $(A, B)$. 

\medskip  We choose an ordering of $\mathcal B$ as follows:
$$\mathcal B=(\a_A, \r_A, \a_B, \r_B,  \x_{1, A}, \ldots, \x_{n-1, A},  \x_{1, B}, \ldots, \x_{n-1, B}).$$
Such an ordering will be called a \emph{canonical ordering}. This gives a tuple of boundary points
$$\mathfrak p=(\p_1, \p_2, \q_1, \q_2, \p_3,\ldots, \p_{n+1},  \q_3, \ldots, \q_{n+1}),$$
where $\p_i$, $\q_i$ are defined by \eqnref{b1}. Note that not all $\p_i$, $\q_i$ may not be distinct. If they are not, say $\p_i=\q_j$, then we replace $\x_{i, A}$ by $\lambda \x_{i, A}$ for some $\lambda \in \s^1$ and choose a different $\q_j$ from the chain spanned by $\frac{\a_A-\r_A}{\sqrt 2}$ and $\x_{i, A}$.  The resulting ordered tuple of distinct points $(p_1, \ldots, p_{2n+2})$ will be called a \emph{reference  ordered tuple of eigenpoints} of $(A, B)$, or simply as \emph{reference eigenpoint} of $(A, B)$. 

\medskip Given $[(A, B)] \in \fD(\F_2, \SU(n,1))$, we choose a  reference eigenpoint $\mathfrak p$ of $(A, B)$ as above using \eqnref{b1}. After fixing such a choice,  the assignment of a reference eigenpoint to a pair $(A, B)$ is well-defined up to the diagonal action of $Z(A) \times Z(B)$ on the eigenframes. Since the diagonal subgroup of $Z(A) \times Z(B)$ (if nontrivial) is a subgroup of $\SU(n,1)$, this associates the point $[\mathfrak p]$ on the space $\M(n, 2n+2)$ to the class $[(A, B)]$. We call this point $[\mathfrak p]$ as the \emph{reference orbit} of $(A, B)$. 

 Now, we define projective invariants following Cunha and Gusevskii \cite{cugu1}.  
\begin{definition}\label{cid}
Let $(A, B)$ be a pair of loxodromics in $\SU(n,1)$.  We fix the canonical  ordering of $\mathcal B$. To a reference eigenpoint $(p_1, \ldots, p_{2n+2})$ of $(A, B)$, we associate the following conjugacy invariants:

\medskip We associate complex numbers $\X_{2j}(A, B)$, $\X_{3j}(A, B)$, $\X_{kj}(A,B)$  given by the following:  
$$\X_{2j}(A, B)={\X(p_1,~p_2,~p_3,~p_j)}, ~  ~{\X}_{3j}(A, B)=\X(p_1,~p_3,~p_2,~p_j),$$ 
$$\X_{kj}(A,B)={\X(p_1,~p_k,~p_2,~p_j)},  $$
where $4 \leq j \leq 2n+2, ~~4 \leq k \leq 2n+2, ~~k < j$.

\medskip The invariants defined above are called \emph{reference cross ratios} of the pair $(A, B)$. It is easy to see that there are $ (n+1)(2n-1)$ non-zero cross ratios in the above list. 

\medskip Finally define the \emph{angular invariant} of $(A, B)$: 
$$\A(A, B)=\A(p_1, p_2, p_3).$$
\end{definition}

The \thmref{mt2} now follows from  \thmref{cugu} and the following.
\begin{theorem}\label{mt1}
Let $(A, B)$ be a loxodromic pair in $\SU(n,1)$. Then the $\SU(n, 1)$ conjugation orbit of $(A, B)$ is determined  by  the traces ${ \tr}({A^i})$, ${\tr}( B^i)$,  $1\leq i \leq \lfloor (n+1)/2 \rfloor$, and its reference orbit on $\M(n, 2n+2)$. 
\end{theorem}

 \begin{proof}Suppose $(A, B)$ and $(A', B')$ are loxodromic pairs with the same traces and the same reference orbit. Following the notation in \secref{sec:lox}, $A=C_{A}E_{A}C_{A}^{-1},~B=C_{B}E_{B}C_{B}^{-1}$ and similarly for $A'$ and $B'$. Since $(A, B)$ and $(A', B')$ defines the same reference orbit, it follows from \lemref{lem:normalization}  that there exists a $C$ in $\SU(n,1)$ such that $C(a_{A})=a_{A'},C(x_{k,A})=x_{k,A'},~C(r_{A})=r_{A'}$,  and, 
 $C(a_{B})=a_{B'},C(x_{k,B})=x_{k,B'},~C(r_{B})=r_{B'}$ for $1\leq k \leq n-1$. Therefore $CAC^{-1}$
 and $A'$ have same eigenvectors. Since $\tr(A')^i=\tr(CAC^{-1})^i\text{ for } 1\leq i\leq \lfloor (n+1)/2 \rfloor$, by \corref{licor3}, we must have $CAC^{-1}=A'$.
 Similarly, $B'=CBC^{-1}$. This completes the proof.   
\end{proof}

\begin{remark}
In view of \lemref{lem:normalization}, it follows from the  above theorem that $\fD(\F_2, \SU(n,1))$ has a projection into  
$$\C^{\lfloor (n+1)/2\rfloor} \times \C^{\lfloor (n+1)/2\rfloor} \times \M(n, 2n+2).$$\end{remark}

\section{ Classification of Loxodromic Pairs} \label{ntol}
Note that in the previous section, the assignment of each class $[(A, B)]$ to a tuple of boundary points depends on the choice of $\mathcal B=(\mathcal B_A, \mathcal B_B)$ and after fixing such a choice, it is independent up to the diagonal action of $Z(A) \times Z(B)$. As an advantage of the construction in the previous section, we could associate numerical conjugacy invariants to  pairs. However, the choice of eigenframes of $(A, B)$ is unique up to an action of the full group $Z(A) \times Z(B)$. If we want to get an explicit description of the moduli space independent of the choice of $\mathcal B$,  we need to consider the unique assignment of the $Z(A) \times Z(B)$-orbit of tuples of boundary points to the pair $(A, B)$. We attempt this in this section. However, we do not know how to associate numerical conjugacy invariants in this approach.

\subsection{Moduli of normalized boundary points} Consider the set  $\mathcal E$ of ordered tuples of boundary and polar points on $(\partial \ch^n)^4 \times \P(\V_+)^{2n-2}$ given by a pair of orthonormal frames $(F_1, F_2)$:  
$$\mathfrak p=(q_1, q_2, r_1, r_2, \ldots, r_{n-1}, q_{n+1}, q_{n+2}, r_{n+3}, \ldots, r_{2n+2}).$$
 This corresponds to pair of orthonormal frames of $\C^{n,1}$: 
$$\hat {\mathfrak p}=(\q_1, \q_2, \r_1, \r_2, \ldots, \r_{n-1}, \q_{n+1}, \q_{n+2}, \r_{n+3}, \ldots, \r_{2n-1}),$$
where $\{ \q_1, \q_2\} \cap \{ \q_{n+1}, \q_{n+2}\} =\phi$,  $\langle \q_i, \q_i \rangle=0=\langle \q_{n+i}, \q_{n+i} \rangle$ for  $i=1, 2$, 
$\langle \r_j, \r_j\rangle=\langle \r_{n+j}, \r_{n+j} \rangle=1$ for all $j=1, \ldots, n-1$,
$\langle \q_1, \q_2 \rangle=\langle \q_{n+1}, \q_{n+2} \rangle=\langle \q_1, \q_{n+2} \rangle=1$, 
$\langle \q_i, \r_j \rangle=0=\langle \q_{n+i}, \r_{n+j} \rangle$,  for $i=1, 2$,  $j=1, \ldots, n-1$. 

\medskip To each such point, we have an ordered tuples of boundary points, not necessarily distinct,  $(p_1, \ldots, p_{2n+2})$ satisfying the conditions:
\begin{equation}\label{eq1} 
\langle \p_1, \p_2 \rangle= \langle \p_{n+2}, \p_{n+3}\rangle=\langle \p_1, \p_{n+2}\rangle=1,\end{equation}
\begin{equation}\label{eq12}
\langle \p_i, \p_j \rangle=-1=\langle \p_{n+i}, \p_{n+j} \rangle, i \neq j, ~i, j=3, \ldots, n-1;
\end{equation}
\begin{equation} \label{eq2} \langle \p_1, \p_i \rangle=-{1}/{\sqrt 2}=\langle \p_{n+2}, \p_{n+i}\rangle,~ i=3, \ldots, n-1;\end{equation}
\begin{equation} \label{eq3}  \langle \p_2, \p_k \rangle={1}/{\sqrt 2}=\langle \p_{n+2}, \p_{n+k} \rangle, ~k=3, \ldots, n-1,\end{equation}
here $\p_i$ denotes the standard lift of $p_i$ for each $i$.  Note that not all of $\p_i, \p_{n+i}$, $i=3, \ldots, n$ may be distinct. If they are not distinct, we relabel them and write them as a ordered tuple of distinct boundary points $\hat {\mathfrak p}=(p_1, p_2, \ldots, p_t)$, $n+3 \leq t \leq 2n+2$, so that they correspond to the original ordering of $\mathfrak p$. 

\medskip Let $\mathcal L_t$ be the section of $\M(n, t)$ defined by the equations \eqnref{eq1}--\eqnref{eq3}, and the ordering as describe above. Let $\mathcal L=\bigcup_{t=n+3}^{ 2n+2} \mathcal L_t$. This space can be viewed as a subspace of $\R^{(2n+2)(2n-1)+1}$, where the embedding of $\M(n,t)$ into the affine space  is defined by the inclusion map:
$$(x_1, \ldots, x_t) \mapsto (x_1, \ldots, x_t, 0, \ldots, 0).$$
Thus it has the the induced topology.

\subsection{Pairs of Loxodromics} 
Let $(A, B)$ be a loxodromic pair in $\SU(n,1)$ with multiplicities $(a_1, \ldots, a_k, b_1, \ldots, b_l)$. 
As in the previous section, we now consider the ordered canonical eigenframe to $(A, B)$ given by the tuple
$$\mathfrak e=(\a_A,  \r_A, \x_{1, A}, \ldots, \x_{n-1, A}, \a_B, \r_B, \x_{1, B}, \ldots, \x_{n-1, B}).$$
with normalization as follows:
\begin{equation} \label{n1}
\langle \a_A,\r_A \rangle=1=\langle \a_B,\r_B \rangle ,~ ~\langle \x_{i, A}, \x_{j, A}\rangle=0=\langle \x_{i, B}, \x_{j, B}\rangle, ~i \neq j; 
\end{equation}
\begin{equation} \label{n2}
\langle \x_{i,A},\x_{i,A}\rangle =1=\langle \x_{i,B},\x_{i,B}\rangle; ~~\langle \r_A,\a_B \rangle=1.
\end{equation}

This defines a point on  $(\partial \ch^n)^4 \times (\P(\V_+))^{2n-2}$ that we refer as \emph{canonical eigenpoint}. We further assign a canonical ordered tuple of boundary points $(\mathcal B_A, \mathcal B_B)$ to $\mathfrak e$ (with canonical ordering) defined by \eqnref{b1} as done in the previous section: 
\begin{equation}\label{ep} \p=(a_A, r_A,  q_{1, A}, \ldots, q_{n-1, A}, a_B, r_B, q_{1, B}, \ldots, q_{n-1, B}). \end{equation}
  Note that in this case,  $\mathfrak e$, and hence $(p_1, \ldots, p_{2n+2})$ is determined by $(A, B)$ up to a right action of the group 
$$G=\C^{\ast} \times \U(a_1) \times \ldots \times \U(a_k) \times \U(b_1) \times \ldots \times \U(b_l)$$
on $\p$ given by the following:  for $g=(\lambda, A_1, \ldots, A_k, B_1, \ldots, B_l)\in G$, 
$$g. \p=\bigg(\lambda\a_A ,  \bar{\lambda}^{-1} \r_A , Y_1, \ldots, Y_k ,\lambda \a_B, {\bar \lambda}^{-1} \r_B, Z_1, \ldots, \ldots, Z_l ),$$
where, $Y_i=(\x_{t_i,A}, \ldots,\x_{t_i+a_i-1, A}) A_i$, 
 $Z_j=(\x_{s_j+1, B}, \ldots, \x_{s_j+b_j-1, B}) B_j$,  $t_i=\sum_{p=1}^{i} a_{p-1}$, $s_j=\sum_{p=1}^{j} b_{p-1}$, $a_0=b_0=1$. 

  The group $G$ represents the group $Z(A) \times Z(B)$ and it acts by the above action on the set consisting of loxodromic pairs with multiplicities $(a_1, \ldots, a_k, b_1, \ldots, b_l)$. So, to each such loxodromic pair  $(A, B)$, we have a  $G$-orbit of canonical tuples of boundary points. Since the action of $Z(A)$ and $Z(B)$ in this case is not necessarily the diagonal action, the orbit may move over  $\M(n, t)$, $n+3 \leq t \leq 2n+2$, and it defines a point on $\mathcal L$.

\medskip  The above action of $G$ on $\mathfrak p$ induces an action of $G$ on $\mathcal L$, and gives a $G$-orbit of $[\mathfrak p]$ in $\mathcal L$. This $G$-orbit $[\mathfrak p]$ on $\mathcal L$ corresponds uniquely to the conjugacy class of $(A, B)$, we call it \emph{canonical orbit} of $(A, B)$.  The orbit space on $\mathcal L$ under the above $G$-action is denoted by $\mathcal{OL}_n(a_1, \ldots, a_k; b_1, \ldots, b_l)$,  or simply, $\mathcal{OL}_n$ when there is no ambiguity on $(a_1, \ldots, a_k, b_1, \ldots, b_l)$. Each point on $\mathcal{OL}_n$ corresponds to a conjugacy class of a loxodromic pair $(A, B)$ with multiplicity $(a_1, \ldots, a_k, b_1, \ldots, b_l)$. When both $A$ and $B$ are regular, i.e. have distinct eigenvalues, we denote this orbit space as $\mathcal {RL}_n$.  In this case, by \lemref{lem:normalization},  $\mathcal{RL}_n$ can be canonically realized as a subspace of $\R^{2n(2n-3)+1}$.  Now we have proof of \thmref{thm2}.
\subsection{Proof of \thmref{thm2}}

\begin{proof}Suppose $(A, B)$ and $(A', B')$ are loxodromic pairs with the same type \\  $(a_1, \ldots, a_k, b_1, \ldots, b_l)$, same traces and the same canonical orbit.
Following the notation in \secref{sec:lox}, $A=C_{A}D_{A}C_{A}^{-1},~B=C_{B}D_{B}C_{B}^{-1}$ and similarly for $A'$ and $B'$. In this case, $C_A$ is an element in the subgroup 
$\U(1,1) \times \U(a_1) \times \ldots \times \U(a_k)$:
$$C_A=\begin{bmatrix} \a_A & E_1 &  & \ldots E_k & \r_A \end{bmatrix},$$
where $E_i=\begin{bmatrix} x_{t_i, A} & \ldots & \x_{t_i +a_i -1, A} \end{bmatrix}$, $D_A$ is the diagonal matrix
$$D_A=\begin{bmatrix} re^{i \theta} & 0 & & \ldots&  \\ 0 & \lambda_1 I_{a_1} & 0 & 0 \ldots & 0\\& & 
\ddots & & \\0 & 0 & 0 & \lambda_k I_{a_k} & 0 \\ 0 & 0 & 0 & 0 & r^{-1} e^{i \theta} \end{bmatrix},$$
here $I_s$ denote identity matrix of rank $s$. Similarly for $C_B$. 

 Since the canonical orbits are equal, by \lemref{lem:normalization} it follows that there exist $C\in{\rm{{\rm SU }}}(n,1)$ such that $C(a_{A})=a_{A'},~C(r_{A})=r_{A'}, ~C(a_{B})=a_{B'}, C(r_B)=r_B$, and for $1 \leq i \leq k$, $ 1\leq j \leq l$, 
$$C(\x_{t_i, A}, \ldots, \x_{t_i+a_i-1, A})= (\x_{t_i, A'}, \ldots, \x_{t_i+a_i-1, A'}) U_i,$$
$$C(\x_{t_j, B}, \ldots, \x_{t_j+b_j-1, B})= (\x_{t_j, B'}, \ldots, \x_{t_j+a_i-1, B'}) V_j,$$
where $U_i \in \U(a_i)$, $V_j \in \U(b_j)$. Let 
$$M=\begin{bmatrix} \lambda & 0 & & \ldots&  \\ 0 & U_1 & 0 & 0 \ldots & 0\\& & 
\ddots & & \\0 & 0 & 0 &  U_k & 0 \\ 0 & 0 & 0 & 0 & \bar \lambda^{-1}  \end{bmatrix},$$
Therefore $$CAC^{-1}=[C(C_A)] D_A [C(C_A)]^{-1}=C_A' MD_A M^{-1} C_A'^{-1}.$$
Observing that $M$ commutes with $D_A$ and $D_A=D_{A'}$, we conclude $CAC^{-1}=A'$. 
 Similarly, $B'=CBC^{-1}$. This completes the proof.   
\end{proof}

\begin{remark}
Let $\fR(\F_2, \SU(n,1))$ be the subset of $\fD(\F_2, \SU(n,1))$ consisting of regular pairs. By \corref{cl1}  it follows that $\fR(\F_2, \SU(n,1))$ is embedded in the topological space $\mathcal T \times \mathcal T \times  \mathcal{RL}_{n}$.  The space $\fD(\F_2, \SU(n,1))$ is embedded in $\C^{\lfloor (n+1)/2\rfloor} \times \C^{\lfloor (n+1)/2\rfloor} \times \mathcal{OL}_n$.
\end{remark} 

\section{Examples of Good Pairs}\label{good}
In this section, we construct two classes of pairs for whom, a suitable chosen normalization of eigenframes reduces the $Z(A) \times Z(B)$ action to a single orbit on some $\mathcal L_t$. Thus, these classes of pairs can be parametrized canonically by conjugacy invariants. 

\subsection{Good Pairs I}
Let $A$ and $B$ be two loxodromic elements without a common fixed point. Consider the loxodromic pair $(A, B)$ that has the following property:  for every positive eigenvector $\y$ of $B$, there exists a positive eigenvector $\x$ of $A$  such that $\langle \x, \y \rangle\neq 0$. We shall further assume $(A, B)$ is regular, i.e. both $A$ and $B$ are regular.    We denote the set of all such pairs in $\fD(\F_2, \SU(n,1))$ as $\mathcal T_1$.

\medskip Let $(A, B)$ represents an element in $\mathcal T_1$.  In this case, we choose a pair of eigenframes $\mathcal B=(\mathcal B_A, \mathcal B_B)$,  normalized and arranged so that for $i=1, \ldots, n-1$, 
 $$\langle \a_A,\r_A \rangle=\langle \a_B,\r_B\rangle=\langle \a_A,\r_B \rangle=1, ~\langle \x_{i,A},\x_{i,B}\rangle =1=\langle \x_{i, A}, \x_{i, A} \rangle.$$
Choose a canonical ordering on $\mathcal B$: 
$$\mathcal B=(\a_A, \r_A, \a_B, \r_B,  \x_{1, A}, \ldots, \x_{n-1, A},  \x_{1, B}, \ldots, \x_{n-1, B}).$$
 This gives a tuple of boundary points as before: 
$$\mathfrak p=(\p_1, \p_2, \q_1, \q_2, \p_3,\ldots, \p_{n+1},  \q_3, \ldots, \q_{n+1}),$$
where $\p_i$, $\q_i$, $i=3, \ldots, n+1$,  are defined by \eqnref{b1}. Note that  $\p_i$, $\q_j$ might not be distinct for $i$, $j$. If they are not, we relabel them and re-arrange according to the canonical ordering of $\mathcal B$.

 In a chosen eigenframe of $(A, B)$, normalized as above,  suppose we change $\a_A$ by $\lambda \a_A$, $\lambda \in \C^{\ast}$. Then $\langle\a_A , \a_B\rangle = 1$, resp. $\langle \a_A, \r_A\rangle=1$,  implies that $\a_B$, resp. $\r_A$,  is scaled by ${\bar \lambda}^{-1}$. The relation $\langle\a_B , \r_B \rangle = 1$ implies $\r_B$ is scaled by $\lambda$.  If we change $\x_{i, A}$ by $\mu_i$, then $\x_{i, B}$ is changed by $\bar \mu_{i}^{-1}$.   This implies that  $\mathcal B$, and hence $(p_1, \ldots, p_t)$ is determined  up to  an action of the group $\T=\C^{\ast} \times \U(1)^{n-1}$ on the set of canonical eigenframes. This action is given by the following:  for $g=(\lambda,\mu_1, \ldots, \mu_{n-1})\in \T$, 
\begin{equation}\label{ga1} g. \p=\bigg(\lambda\a_A , \bar{\lambda}^{-1}\r_A , \lambda\a_B , \r_B \bar{\lambda}^{-1}, \mu_1\x_{1, A} , \ldots,  {\mu_{n-1}}\x_{n-1, A},  \mu_1\x_{1, B}, \ldots, {\mu_{n-1}}\x_{n-1, B}  \bigg).\end{equation} 
 Hence to each $(A, B),$ we assign a unique $\T$-tuple of boundary points $(p_1, \ldots, p_t)$. Since $\T$ projects to a subgroup of $\SU(n,1)$, this gives an assignment of the $\SU(n,1)$-conjugation orbit of $(A, B)$ to a unique orbit $[(p_1, \ldots, p_t)]$ in $\mathcal  L_t$. In this case, the number of numerical conjugacy invariants defined as in \defref{cid} depends on $t$.  
\subsection{Good Pairs II} Now we define another class of pairs that generalizes the generic elements we classified in \cite{gp2}. Following the notion in \cite{gp2}, we will call them as `non-singular' here. In the following $L_A$ is the $(1,1)$-subspace as given in \secref{espa}. 

 \begin{definition}\label{ns}
A pair of loxodromics $(A, B)$ is called \emph{non-singular} if 
 \begin{enumerate}
  \item { $A$ and $B$ does not have a common fixed point.  }

  \medskip \item {$\x_{k,A}\notin L_B^{\perp},\, \x_{k,B}\notin L_{A}^{\perp}$ where $k$ ranges over $n-2$ numbers in $\{1, \ldots, n-1\}$,  i.e. for each such $k$, $x_{k, A}$ has non-zero projection on $L_B$ and, $x_{k, B}$ has non-zero projection on $L_A$. Given a non-singular   pair $(A, B)$, without loss of generality, re-arranging the eigenvectors if necessary, we shall assume that $1 \leq k \leq n-2$.  }
\end{enumerate}
\end{definition}

We shall consider regular non-singular  pairs in the following. Unless otherwise specified, a non-singular  pair will always assumed to be regular. Further, we will always assume, by suitably relabeling the eigenvectors if required, that $\langle \x_{i,A},\a_{B}\rangle\neq 0,\, \langle \x_{i,B},\a_A\rangle \neq 0\text{ for }1\leq i\leq n-2$.
\subsubsection{Eigenpoints of a non-singular   pair} \label{nor}
Let $(A, B)$ be {non-singular  } in $\SU(n,1)$. Without loss of generality,  we may assume that $\langle \x_{i,A},\a_{B}\rangle\neq 0,\, \langle \x_{i,B},\a_A\rangle \neq 0\text{ for }1\leq i\leq n-2$ for any choice of lifts. We choose normalized eigenframes such that
$$
\langle \a_A,\r_A \rangle=\langle \a_B,\r_B\rangle=\langle \a_A,\r_B \rangle=\langle \x_{i,A},\x_{i,A}\rangle =\langle \x_{i,B},\x_{i,B}\rangle= 1,
$$
 $$ \langle \x_{i,A},\a_{B}\rangle,\, \langle \x_{i,B},\a_A\rangle \in \mathbb{R}_{+}\text{ for }1\leq i\leq n-2, ~\langle \x_{1,A},\a_{B}\rangle=1. $$

\bigskip To see that this is possible, suppose for some choice of lifts, 
$$\langle \a_A,\r_A \rangle=\lambda,\,\,\langle \a_B,\r_B\rangle=\mu,\langle \a_A,\r_B \rangle=\nu,\,\,\langle \x_{i,A},\x_{i,A}\rangle=r_i^2,\,\langle \x_{i,B},\x_{i,B}\rangle=s_i^2 ,$$ 
 $$ \langle \x_{i,A},\a_{B}\rangle=\gamma_i,\,\, \langle \x_{i,B},\a_A\rangle=\delta_i ,\text{ where }r_i,s_i\in \mathbb{R}_{+}\text{ for }1\leq i\leq n-2.$$
Let us choose the appropriate lifts in the following way.
\begin{enumerate}
\item First replace $\x_{1,A}$ by $r_1^{-1}\x_{1,A}$, so that $\langle\x_{1,A},\x_{1,A}\rangle=1.$

\item Replace $\a_B$ by $r_1{\bar \gamma_1}^{-1}\a_B$, so that $\langle \x_{1,A},\a_B \rangle =1.$

\item Replace $\r_B$ by $r_1^{-1}\gamma_1{\bar \mu}^{-1}\r_B$, so that $\langle \a_B,\r_B \rangle=1.$

\item Replace $\a_A$ by $r_1{\bar \gamma_1}^{-1}\mu\nu^{-1}\a_A $, so that $\langle \a_A,\r_B \rangle=1.$

\item Replace $\r_A$ by $r_1^{-1}\gamma_1{\bar \lambda}^{-1} {\bar \mu}^{-1} {\bar \nu} \r_A $, so that $\langle \a_A,\r_A \rangle=1.$

\item For $i\neq 1$, replace $\x_{i,A}$ by  $r_i^{-1}e^{i(\text{arg}\gamma_1-\text{arg}\gamma_i)}\x_{i,A}$, so that $\langle \x_{i,A},\x_{i,A}\rangle=1, \\ \text{ and } \langle \x_{i,A},\a_{B}\rangle\in \mathbb{R}_{+}$.

\item For $1\leq i \leq n-1$, replace  $\x_{i,B}$ by $ s_i^{-1}e^{i(\text{arg}\gamma_1+\text{arg}\mu-\text{arg}\nu-\text{arg}\delta_i)}\x_{i,B}$, so that $\langle \x_{i,B},\x_{i,B}\rangle=1, \text{ and } \langle \x_{i,B},\a_A\rangle\in \mathbb{R}_{+}.$
\end{enumerate}

\medskip 
With this normalization, we associate to $(A, B)$ an eigenpoint as in \secref{tol}. We denote it  by $\mathfrak p=(p_{1, A}, \ldots, p_{n, A}, p_{1, B}, \ldots, p_{n, B})$. Note that by regularity, we can ignore the $p_{n+1, A}$ and $p_{n+1, B}$ by \lemref{lem:normalization}.  Note that because of non-singularity, $\p_{i, B}$ can not be equal to $\a_A$, $\r_A$ for all $i$, and similarly, $\p_{i, A}$ can not be equal to $\a_B, ~\r_B$. If some $p_{i, A}$ is equal to $p_{j, B}$, we re-arrange them as before and denote by $(p_1, \ldots, p_t)$. 
\medskip 
\begin{lemma}\label{wd}
Let $(A, B)$ be non-singular   pair in ${\rm SU}(n, 1)$. Suppose that $(p_1, \ldots, p_t)$,\\$(p_1', \ldots, p_t')$
are two tuples of eigenpoints of $(A, B)$. Then $\p_i'=\lambda \p_i$ for some $\lambda\in \mathbb{C}$ with $|\lambda|=1$, $1 \leq i \leq t$. 
\end{lemma}

\begin{proof}
By symmetry, it is enough to prove that if either of $\a_A , \x_{i,A}$ is scaled by $\lambda$, then the normalized  eigenframes are scaled by $\lambda$ and $|\lambda|=1$. First consider the case when $\a_A$ is scaled by $\lambda$.
Then $\langle\a_A , \r_B\rangle = $1 implies that $\r_B$ is scaled by ${\bar \lambda}^{-1}$ . Then $\langle\a_B , \r_B \rangle = 1$ implies $\a_B$ is scaled by $\lambda$. Then $\langle \x_{1,A} ,\a_B\rangle = 1$ implies $\x_{1,A}$ is scaled by ${\bar \lambda}^{-1}$ . Then $\langle \x_{1,A} ,\x_{1,A}\rangle = 1$  implies that $|\lambda|=1$
and so ${\bar \lambda}^{-1}=\lambda$. Then the choice
$$\langle\x_{i,A} , \a_B\rangle, \langle \x_{i,B}, \a_A\rangle \in \mathbb{R}_{+} \text{ for }1 \leq i \leq n-2, $$
implies that $\x_{i,A}$ and $\x_{i,B}$ are scaled by $\lambda$. This proves the lemma.\end{proof}

This shows that a non-singular pair in $\fD(\F_2, \SU(n,1))$  not only projects down to a unique point on $\mathcal L_t$, but to each non-singular pair $(A, B)$ of $\SU(n,1)$, there is a unique tuple of boundary points on $\partial \ch^n$. \\

\begin{acknowledgement}
We thank Michael Cowling, John Parker, Richard Schwartz and Stephan Tillmann  for their comments on this work. 
 
\medskip The work was completed when Gongopadhyay was visiting the University of New South Wales, Sydney. Gongopadhyay thanks the UNSW for hospitality and the Indian \hbox{National} Science Academy (INSA) for supporting the visit by an INSA Indo-Australia EMCR Fellowship.  

\medskip  Parsad acknowledges support from IISER Bhopal project, no. INST/MATH/2017028,  during the course of this work. 
\end{acknowledgement}


\end{document}